\documentclass[11pt]{article}
\usepackage{amsfonts}
\usepackage{amsfonts}
\usepackage{amsfonts}
\usepackage{amsfonts}
\usepackage{amsfonts}
\usepackage{amsfonts}
\usepackage{amsfonts}
\usepackage{amsfonts}
\usepackage{amsfonts}
\usepackage{amsfonts}
\usepackage{amsfonts}
\usepackage{amsfonts}
\usepackage{amsfonts}
\usepackage{amsfonts}
\usepackage{amsfonts}
\usepackage{amsfonts}
\usepackage{amsfonts}
\usepackage{mathrsfs}
\usepackage{amssymb,mathrsfs,amsmath,amsthm,amsfonts}
\usepackage{graphicx}
\usepackage[pdfstartview=FitH]{hyperref}
\providecommand{\U}[1]{\protect\rule{.1in}{.1in}}
\allowdisplaybreaks 

\newtheorem{theorem}{Theorem}[section]
\newtheorem{lem}{Lemma}[section]
\newtheorem{prp}[theorem]{Proposition}
\newtheorem{thm}[theorem]{Theorem}

\newtheorem{dfn}[theorem]{Definition}

\newtheorem{remark}{Remark}[section]

\numberwithin{equation}{section}

\title{\bf{Exponential Convergence for Semilinear SDEs Driven by L\'evy Processes on Hilbert Spaces }}
\author{\ Yulin Song$^{1}$\\
{\em\small School of Mathematical Sciences, Beijing Normal University, Beijing, {\rm 100875}, P.R.China}\\
\ Tiange Xu$^{2}$\\
{\em\small School of Mathematics, University of CAS, Beijing, {\rm 100049}, P.R.China}}

\date{}
\begin{document}
\footnotetext[1]{ Corresponding author. Email: songyl@amss.ac.cn.}
\footnotetext[2]{ The author is financially supported by NFS (Grant No. 11101419). Email: xutiange@ucas.ac.cn.}
\maketitle

\begin{abstract}
\noindent In this paper, an integration by parts formula was derived for jump processes on Hilbert spaces. Using this formula, we
investigated derivative formula and exponential convergence for semilinear SDEs driven by purely jump processes on Hilbert spaces.

\vspace{1mm}
\end{abstract}
\section{Introduction}
In recent years, integration by parts formula, derivative formula and exponential convergence have been extensively studied for stochastic
differential equations with jumps. The topics on the two above formulas for jump processes on $\mathbb{R}^d$, we refer to \cite{Bally2007},
\cite{Bass1986}, \cite{Dong2012}, \cite{Norris1988}, \cite{Song2012}, \cite{Takeuchi2010},\cite{Wang2011a},\cite{Wang2011b}, \cite{ZhangXC}
and so on. In the finite-dimensional case, the authors investigated the coupling property
for linear SDEs in \cite{Bottcher2011}, \cite{Schilling2011}, \cite{Schilling2011a}, \cite{Schilling2012} and references within.
For nonlinear SDEs driven by jump processes, coupling property was derived in \cite{Song2012}. In most of references mentioned above
the shift-invariance of the Lebesgue measure plays an essential role. But in infinite-dimensional setting, there is no Lebesgue measure available.
The authors in \cite{Wang2011c} investigated the strong Feller and coupling properties for transition semigroups of linear SDEs driven by
L\'evy processes on a Banach space equipped with a nice reference measure, which has quasi-invariance property.
Exponential convergence for SDEs driven by L\'evy processes was studied in  \cite{Kulik2009},
\cite{Priola2012},\cite{JWang2008} and references therein for finite-dimensional case. Ergodicity and exponential mixing properties of SPDEs
driven by cylindrical stable processes were derived in \cite{Priola2012}, \cite{Priola2011a} and \cite{Priola2011b}.
In this paper we aim to investigate integration by parts formula, derivative formula and exponential convergence for semilinear SDEs driven by
non-cylindrical purely jump processes on a separable Hilbert space.

Let $(\mathbb{H}, \langle\cdot,\cdot\rangle)$ be a separable Hilbert space and $\mu$ be a Gaussian measure on $\mathbb{H}$
with covariance operator $Q$, which is nonnegative, symmetric and has finite trace. Its square root, denoted by $Q^{\frac{1}{2}}$,
is a nonnegative and symmetric Hilbert-Schmidt operator.
Let Im$Q^{\frac{1}{2}}$ be the image space of $Q^{\frac{1}{2}}$, i.e., Im$Q^{\frac{1}{2}}=\{Q^{\frac{1}{2}}x|x\in \mathbb{H}\}$. As is known,
Im$Q^{\frac{1}{2}}$ is a Hilbert space with the induced inner product
\begin{align*}
\langle x,y\rangle_0:=\langle Q^{-\frac{1}{2}}x,Q^{-\frac{1}{2}}y\rangle,~~x,y \in \text{Im}Q^{\frac{1}{2}},
\end{align*}
where $Q^{-\frac{1}{2}}$ is the pseudo inverse of $Q$ in the case that it is not one-to-one, that is, for $h\in$Im$Q^{\frac{1}{2}}$,
\begin{align*}
Q^{-\frac{1}{2}}h=x,\text{ if }Q^{\frac{1}{2}}x=h\text{ and }\|x\|=\inf\{\|y\|:Q^{\frac{1}{2}}y=h\}.
\end{align*}
Equivalently, we also have
\begin{align*}
\langle x,y\rangle_0=\sum_{k=1}^\infty\frac{\langle x,e_k\rangle\langle y,e_k\rangle}{\lambda_k}I_{[\lambda_k>0]},
\end{align*}
where $\{e_k\}_{k\in\mathbb{N}}$, the eigenvectors of $Q$ with eigenvalues $\{\lambda_k\}_{k\in \mathbb{N}}$, consists of an orthonormal basis
of $\mathbb{H}$. The space (Im$Q^{\frac{1}{2}}$,$\langle\cdot,\cdot\rangle_0$) is called the Reproducing Kernel Hilbert Space of $\mathbb{H}$.
As is known, the Gaussian measure $\mu$ has quasi-invariant property under the shift $z\mapsto z+h$ for any $h\in$Im$Q^{\frac{1}{2}}$, i.e.,
$\mu(\cdot+h)$ and $\mu$ are mutually absolutely continuous. The Randon-Nikodym derivative of $\mu(\cdot+h)$ with respect to $\mu$ is
\begin{align*}
\varphi(z,h):=\frac{\mu(dz+h)}{\mu(dz)}=\exp\{\langle h,z\rangle_0-\frac{1}{2}\langle h,h\rangle_0\},~~\mu-a.s.
\end{align*}

For $i=1,2$, let $W_i$ be the space of all c$\grave{a}$dl$\grave{a}$g functions from $[0,\infty)$ to $\mathbb{H}$ vanishing at 0,
which is endowed with the Skorohod topology and the probability measure $\mathbb{P}^i$ such that the coordinate
process $L^i_t(w_i)=w_i(t)$ is a L\'evy process. Furthermore, we assume that $L^1:=\{L^1_t\}_{t\geq0}$ is a purely
jump process with L\'evy measure $\rho(z)\mu(dz)$, where $\rho:\mathbb{H}\rightarrow(0,\infty)$ is a differentiable function
satisfying $\int_\mathbb{H}(|z|^2\wedge1)\rho(z)\mu(dz)<\infty$.

Consider the following product probability space
\begin{align*}
(\Omega, \mathscr{F}, \mathbb{P}):=(W_1\times W_2,\mathscr{B}(W_1)\times\mathscr{B}(W_2), \mathbb{P}^1\times \mathbb{P}^2)
\end{align*}
and define
\begin{align*}
L_t=L^1_t+L^2_t.
\end{align*}
That is, for $w=(w_1,w_2)\in\Omega$,
\begin{align*}
L_t(w)=w_1(t)+w_2(t).
\end{align*}
Then $\{L_t\}_{t\geq0}$ is a L\'evy process on $(\Omega, \mathscr{F}, \mathbb{P})$ with two independent parts and its L\'evy measure denoted by
$\nu$ satisfies $\nu(dz)\geq\rho(z)\mu(dz)$. Denote by $\{\mathscr{F}_t\}_{t\geq0}$ the smallest filtration generated by $\{L_t\}_{t\geq0}$.
We use $N^1$ and $\widetilde{N^1}$ to be the jump measure and martingale measure of $\{L^1_t\}_{t\geq0}$.
Let $\mathbb{E}$ and $\mathbb{E}^1$ be the associated expectations of $P$ and $P^1$ respectively.

In this paper, we consider the following stochastic equation on $\mathbb{H}$
\begin{equation}\label{2}
\begin{cases}
dX_t=AX_tdt+F(X_t)dt+dL_t,\\
X_0=x,
\end{cases}
\end{equation}
where $A: \mathcal{D}(A)\subset\mathbb{H}\rightarrow \mathbb{H}$ is an adjoint, unbounded and linear operator generating
a $C_0$-semigroup $\{S_t\}_{t\geq0}$ on $\mathbb{H}$ and $F: \mathbb{H}\rightarrow \mathbb{H}$ is measurable and bounded.
The mild solution of Eq.($\ref{2}$), if it exists, can be formulated as
\begin{align*}
X_t=S(t)x+\int_0^tS(t-s)F(X_s)ds+\int_0^tS(t-s)dL_s,~~t\geq0.
\end{align*}
We gather here all hypotheses which will be made on Eq.($\ref{2}$).\\
({\bf{H1}}) There exists a differentiable function $\rho:\mathbb{H}\rightarrow(0,\infty)$ with bounded derivative such that
\begin{align*}
\lambda:=\int_\mathbb{H}\rho(z)\mu(dz)<\infty \text{ and } \int_\mathbb{H}|z|^2\rho(z)\mu(dz)<\infty.
\end{align*}
({\bf{H2}}) Im$S(t)\subset$Im$Q$ holds for any $t>0$.\\
({\bf{H3}}) $A$ is a dissipative operator defined by
\begin{align}
A=\sum\limits_{k\geq1}(-\gamma_k)e_k\otimes e_k,
\end{align}
for $0<\gamma_1\leq\gamma_2\leq\cdots\leq\gamma_k\leq\cdots$ and $\gamma_k\rightarrow\infty$ as $k\rightarrow\infty$.\\
({\bf{H4}}) $F:\mathbb{H}\rightarrow\mathbb{H}$ is a bounded and Lipschitz continuous function with the smallest Lipschitz constant $\|F\|_{Lip}$.

Let $\mathscr{B}_b(\mathbb{H})$ be the class of all bounded measurable function on $\mathbb{H}$.
We use $C_b^2(\mathbb{H})$ to denote the family of $C^2$ real-valued functions $f$ such that $f$ and its derivatives of order up to 2 are bounded.
$\|\cdot\|$, $\|\cdot\|_{\infty}$ and $\|\cdot\|_{Var}$ stand for the operator norm, uniform norm and total variation norm respectively.
Denote the associated jump measure of $L^1$ by $N^1$ and the counting process by $N^1_t$, that is,
$N^1_t:=N^1([0,t]\times\mathbb{H})$.
We aim to derive the Bismut type formula for semigroups $\{P_t^1\}_{t\geq0}$ and the exponential convergence for $\{P_t\}_{t\geq0}$
defined as
\begin{align}
P^1_tf(x):=\mathbb{E}\{f(X_t^x)I_{[N^1_t\geq1]}\},~~~~
P_tf(x):=\mathbb{E}f(X_t^x),~~~~x\in\mathbb{H},~t\geq0,~f\in\mathscr{B}_b(\mathbb{H}).
\end{align}
In this paper, it is difficult to investigate a derivative formula for $P_t$. Fortunately, associated formula for
$P^1_t$ is derived, which is enough for our further work.

Let $J_t\xi$ be the derivative of $X^x_t$ w.r.t. the initial value $x$. Denote
\begin{align*}
C_b^1(\mathbb{H}\rightarrow\mathbb{H})=&\Big\{G:\mathbb{H}\rightarrow\mathbb{H}\big|
G\text{ is bounded, differential, with bounded} \\&\text{ ~~~~~~~~~~~~~~~~~~and continuous derivatives up to order }1.\Big\}.
\end{align*}

We have the following main results.
\begin{thm}
Assume $(\bf{H1})-(\bf{H2})$ hold. Let $F\in C^1_b(\mathbb{H}\rightarrow\mathbb{H})$ and $\nabla F$ be Lipschitz continuous.
If $\int_0^t\|Q^{-1}S(s)\|ds<\infty$, then for $f\in C^2_b(\mathbb{H})$ and $\xi\in\mathbb{H}$
\begin{align*}
\nabla_\xi P_t^1f(x)=-\mathbb{E}\Big\{f(X_t^x)\frac{I_{[N^1_t\geq1]}}{N^1_t}
\int_0^t\int_\mathbb{H}\Big(\langle z, Q^{-1}J_s\xi\rangle+\langle\nabla\log\rho(z), J_s\xi\rangle\Big)\widetilde{N^1}(dz,ds)\Big\}.
\end{align*}
\end{thm}
As a direct application of the formula, exponential convergence of the transition semigroup is derived.
\begin{thm}
Assume $(\bf{H1})-(\bf{H4})$ hold. If $\gamma_1>\|F\|_{Lip}$ and $\lim\limits_{t\rightarrow\infty}\frac{\int_0^t\|Q^{-1}S(s)\|^2ds}{t}<\infty$,
then there exists a constant $C>0$ such that
\begin{align}
\label{0}
\|P_t(x,\cdot)-P_t(y,\cdot)\|_{Var}\leq C(1+|x-y|)\exp\Big\{-\frac{\lambda(\gamma_1-\|F\|_{Lip})t}{\lambda+\gamma_1-\|F\|_{Lip}}\Big\}.
\end{align}
\end{thm}
\begin{remark}
Under the conditions of Theorem 1.1, if the solution admits invariant probability measures and each invariant measure is integrable,
then exponential ergodicity follows from $(\ref{0})$.
\end{remark}

An example is given to show the conditions of Theorem 1.1 and Theorem 1.2 on $Q$ and $S(t)$ can be satisfied.\\
{\bf{Example:}} Let $(\bf{H3})$ hold. For $0<\delta<\frac{1}{2}$, the fractional power $(-A)^\delta$ of $-A$ is defined by
\begin{align*}
(-A)^\delta=\frac{1}{\Gamma(\delta)}\int_0^\infty t^{-\delta}S(t)dt,
\end{align*}
where $\Gamma$ is the Euler function. It can be proved that $S(t)\mathbb{H}\subset\mathcal{D}((-A)^\delta)$ and for any $t>0$,
\begin{align*}
\|(-A)^\delta S(t)\|\leq C_\delta t^{-\delta}
\end{align*}
for a suitable positive constant $C_\delta$. Take $Q=\big((-A)^\delta\big)^{-1}$, then we have
$S(t)\mathbb{H}\subset$Im$Q$. Moreover,
\begin{align*}
\lim\limits_{t\rightarrow\infty}\frac{\int_0^t\|Q^{-1}S(s)\|^2ds}{t}\leq\lim\limits_{t\rightarrow\infty}
\frac{C^2_\delta\int_0^ts^{-2\delta}ds}{t}=\lim\limits_{t\rightarrow\infty}\frac{C^2_\delta t^{1-2\delta}}{t}=0.
\end{align*}

The rest of this paper is organized as follows: in section 2, we shall investigate an integration by parts formula for jump processes
valued on $\mathbb{H}$; the proofs of the main results will be presented in section 3.
\section{Integration by Parts Formula}
An integration by parts formula can enable one to derive the derivative formula, and it is a powerful tool in stochastic analysis.
The associated formula for jump processes on $\mathbb{R}^d$ can be found in \cite{Bass1986}, \cite{Dong2012}, \cite{Norris1988},\cite{Takeuchi2010}
 and so on. But so far, there are few references studying the formula for jump processes in infinite-dimensional case.

\noindent For fixed $T>0$, denote
\begin{align*}
L^1(\Omega\times[0, T])=\Big\{V:\Omega\times[0, T]\rightarrow \mathbb{H}\Big|\int_0^T\mathbb{E}|V(s)|ds<\infty\Big\}.
\end{align*}
Define a perturbed random measure $N^{1,\epsilon}$ by
\begin{align}
N^{1,\epsilon}(\Gamma\times [0,t])=\int_0^t\int_{\mathbb{H}}I_\Gamma(z+\epsilon V(s))N^1(dz,ds),~~~\Gamma\in\mathscr{B}(\mathbb{H}),
\end{align}
for $V\in L^1(\Omega\times[0, T])$.
Let $\{L_t^{1,\epsilon}\}_{0\leq t\leq T}$ be the associated L\'evy process perturbed by $\epsilon V$, that is ,
\begin{eqnarray*}
L_t^{1,\epsilon}=L^1_t+\epsilon\int_0^t\int_{\mathbb{H}}V(s)N^1(dz,ds).
\end{eqnarray*}
Recall the following notion of $L^1$-derivative, which was first introduced in \cite{Bass1986} and can also be found
in \cite{Dong2012} and \cite{Song2012}.
\begin{dfn}
A functional $G_t=G(\{L_s\}_{s\leq t})$ is called to have an $L^1$-derivative in the direction $V$,
if there exists an integrable random variable denoted by $D_VG_t$, such that
\begin{align*}
\lim\limits_{\epsilon\rightarrow0}\mathbb{E}
\Big|\frac{G(\{L^\epsilon_s\}_{s\leq t})-G(\{L_s\}_{s\leq t})}{\epsilon}-D_VG_t\Big|=0.
\end{align*}
\end{dfn}
In order to obtain the integration by parts formula, we would like to construct a weighted probability measure
such that the distribution of the perturbed process under this weighted probability equals the one of original
processes under the reference measure $P^1$. For the sake of convenience, denote
\begin{align*}
\mathscr{V}_1=\Big\{V:\Omega\times[0,T]\rightarrow \text{Im}Q\Big|V\text{ is predictable and }
\sup\limits_{s\leq T}|Q^{-1}V(s)|<\infty.\Big\},
\end{align*}
and
\begin{align*}
\mathscr{V}_2=\Big\{V:\Omega\times[0,T]\rightarrow \text{Im}Q\Big|V\text{ is predictable and }
\int_0^T\mathbb{E}|Q^{-1}V(s)|ds<\infty.\Big\}.
\end{align*}
For $\epsilon>0$, $V\in\mathscr{V}_1$ and $w_2\in\Omega_2$, we set
\begin{align*}
\lambda^{\epsilon,w_2}(s,z)=
\varphi\big(z,\epsilon V^{w_2}(s)\big)\frac{\rho\big(z+\epsilon V^{w_2}(s)\big)}{\rho(z)}.
\end{align*}
Take
\begin{eqnarray*}
Z_t^{\epsilon,w_2}=\exp\Big\{\int_0^t\int_\mathbb{H}\log\lambda^{\epsilon,w_2}(s,z)N^1(dz,ds)
-\int_0^t\int_\mathbb{H}(\lambda^{\epsilon,w_2}(s,z)-1)\rho(z)\mu(dz)ds\Big\},
\end{eqnarray*}
then $\{Z_t^{\epsilon,w_2}\}_{t\leq T}$ satisfies the following equation
\begin{equation*}
\begin{cases}
dZ_t^{\epsilon,w_2}=Z_{t-}^{\epsilon,w_2}(\lambda^{\epsilon,w_2}(t,z)-1)\widetilde{N^1}(dz,dt),\\
Z_0^{\epsilon,w_2}=1,
\end{cases}
\end{equation*}
and can be formulated as
\begin{align}
\label{3}Z_t^{\epsilon,w_2}=1+\int_0^t\int_\mathbb{H}Z_{s-}^{\epsilon,w_2}(\lambda^{\epsilon,w_2}(s,z)-1)\widetilde{N^1}(dz,ds).
\end{align}
So
$\{Z_t^{\epsilon,w_2}\}_{0\leq t\leq T}$ is a martingale and ${Z_0^{\epsilon,w_2}=1}$.
Define a probability measure $\mathbb{P}^{1,\epsilon,w_2}$ as
\begin{align*}
\frac{d\mathbb{P}^{1,\epsilon,w_2}}{d\mathbb{P}^1}\Big|_{\mathcal {F}_t}=Z_t^{\epsilon,w_2},~~t\leq T.
\end{align*}
\begin{lem}
For any $V\in\mathscr{V}_1$, the law of $L^{1,\epsilon}_t$ under $\mathbb{P}^{1,\epsilon,w_2}$ is equal to the one of $L_t^1$ under $\mathbb{P}^1$.
\end{lem}
\begin{proof}
For any test function $\phi\leq0$ and fixed $w_2\in\Omega_2$, denote
\begin{align*}
Y_t^{\epsilon,w_2}=\exp\Big\{\int_0^t\int_{\mathbb{H}}\phi(s,z)N^{1,\epsilon,w_2}(dz,ds)\Big\},\ \ \
\ \ G_t^{\epsilon,w_2}=Y_t^{\epsilon,w_2} Z_t^{\epsilon,w_2}.
\end{align*}
We just check that $\mathbb{E}^{1,\epsilon,w_2}G_t^{\epsilon,w_2}$ does not depend on $\epsilon$,
where $\mathbb{E}^{1,\epsilon,w_2}$ denotes the associated expectation of conditional probability
$\mathbb{P}^{1,\epsilon,w_2}$. Actually, note that
\begin{align*}
Y_t^{\epsilon,w_2}=&\exp\Big\{\int_0^t\int_{\mathbb{H}}\phi(s,z)N^{1,\epsilon,w_2}(dz,ds)\Big\}\\
=&\exp\Big\{\int_0^t\int_{\mathbb{H}}\phi(s,z+\epsilon V^{w_2}(s))N^1(dz,ds)\Big\}\\
=&1+\int_0^t\int_{\mathbb{H}}Y_{s-}^{\epsilon,w_2}\big(e^{\phi(s,z+\epsilon V^{w_2}(s))}-1\big)N^1(dz,ds),
\end{align*}
and
\begin{align*}
[Y^{\epsilon,w_2},Z^{\epsilon,w_2}]_t
=\int_0^t\int_\mathbb{H}Y_{s-}^{\epsilon,w_2} Z_{s-}^{\epsilon,w_2}\big(e^{\phi(s,z+\epsilon V^{w_2}(s))}-1\big)
\big(\lambda^{\epsilon,w_2}(s,z)-1\big)N^1(dz,ds).
\end{align*}
Applying It$\hat{\mathrm{o}}$ formula, it yields
\begin{align*}
G_t^{\epsilon,w_2}=&1+\int_0^t\int_\mathbb{H}Y_{s-}^{\epsilon,w_2}dZ_s^{\epsilon,w_2}
+\int_0^t\int_\mathbb{H}Z_{s-}^{\epsilon,w_2} dY_s^{\epsilon,w_2}+[Y^{\epsilon,w_2},Z^{\epsilon,w_2}]_t\\
=&1+\int_0^t\int_\mathbb{H}Y_{s-}^{\epsilon,w_2}dZ_s^{\epsilon,w_2}
+\int_0^t\int_\mathbb{H}Z_{s-}^{\epsilon,w_2}Y_{s-}^{\epsilon,w_2}\big(e^{\phi(s,z+\epsilon V^{w_2}(s))}-1\big)N^1(dz,ds)\\
&+\int_0^t\int_\mathbb{H}Y_{s-}^{\epsilon,w_2}Z_{s-}^{\epsilon,w_2}\big(e^{\phi(s,z+\epsilon V^{w_2}(s))}-1\big)
\big(\lambda^{\epsilon,w_2}(s,z)-1\big)N^1(dz,ds)\\
=&1+\int_0^tY_{s-}^{\epsilon,w_2}dZ_s^{\epsilon,w_2}+\int_0^t\int_\mathbb{H}G_{s-}^{\epsilon,w_2}\big(e^{\phi(s,z+\epsilon
V^{w_2}(s))}-1\big)\lambda^{\epsilon,w_2}(s,z)N^1(dz,ds).
\end{align*}
Furthermore,
\begin{align*}
\mathbb{E}^1G_t^{\epsilon,w_2}
=&1+\mathbb{E}^1\int_0^t\int_\mathbb{H}G_{s}^{\epsilon,w_2}\big(e^{\phi(s,z+\epsilon V^{w_2}(s))}-1\big)
 \varphi(z, \epsilon V^{w_2}(s))\frac{\rho(z+\epsilon V^{w_2}(s))}{\rho(z)}\rho(z)\mu(dz)ds\\
=&1+\mathbb{E}^1\int_0^t\int_\mathbb{H}G_{s}^{\epsilon,w_2}\big(e^{\phi(s,z+\epsilon V^{w_2}(s))}-1\big)
\varphi(h,\epsilon V^{w_2}(s))\rho(z+\epsilon V^{w_2}(s))\mu(dz)ds\\
=&1+\mathbb{E}^1\int_0^t\int_\mathbb{H}G_{s}^{\epsilon,w_2}\big(e^{\phi(s,z+\epsilon V^{w_2}(s))}-1\big)
\rho(z+\epsilon V^{w_2}(s))\mu\big(dz+\epsilon V^{w_2}(s)\big)ds\\
=&1+\int_0^t\mathbb{E}^1G_s^{\epsilon,w_2}\int_\mathbb{H}\big(e^{\phi(s,z)}-1\big)\rho(z)\mu(dz)ds.
\end{align*}
Therefore,
\begin{align*}
\mathbb{E}^1G_t^{\epsilon,w_2}=\exp\Big\{\int_0^t\int_\mathbb{H}\big(e^{\phi(s,z)}-1\big)\rho(z)\mu(dz)ds\Big\}.
\end{align*}
\end{proof}

\begin{lem}
Assume $({\bf{H1}})$ holds. If there exists a constant $\delta>0$ such that $\rho(z)\geq\delta$ for $\forall z\in\mathbb{H}$,
then for any fixed $t\in[0,T]$ and $V\in\mathcal{V}$,
\begin{align*}
\sup\limits_{\epsilon\leq1}\mathbb{E}^1\Big\{\Big|\frac{Z^{\epsilon,w_2}_t-1}{\epsilon}\Big|^2\Big\}<\infty.
\end{align*}
for any fixed $w_2\in\Omega_2$.
\end{lem}
\begin{proof}
For fixed $\epsilon\in(0,1)$, it follows $(\ref{3})$ that
\begin{align*}
Z_t^{\epsilon,w_2}-1=\int_0^t\int_\mathbb{H}Z^{\epsilon,w_2}_{s-}(\lambda^{\epsilon,w_2}(s,z)-1)\widetilde{N^1}(dz,ds).
\end{align*}
Triangle inequality and B-D-G inequality yield
\begin{align}
\label{1}
\mathbb{E}^1\Big\{\sup\limits_{s\leq t}\Big|\frac{Z^{\epsilon,w_2}_s-1}{\epsilon}\Big|^2\Big\}
=&\mathbb{E}^1\Big\{\sup\limits_{s\leq t}\Big|\int_0^s\int_\mathbb{H}
\frac{Z^{\epsilon,w_2}_{r-}(\lambda^{\epsilon,w_2}(r,z)-1)}{\epsilon}\widetilde{N^1}(dz,dr)\Big|^2\Big\}\cr
\leq&2\mathbb{E}^1\Big\{\sup\limits_{s\leq t}\Big|\int_0^s\int_\mathbb{H}
\frac{Z^{\epsilon,w_2}_{r-}-1}{\epsilon}(\lambda^{\epsilon,w_2}(r,z)-1)\widetilde{N^1}(dz,dr)\Big|^2\Big\}\cr
&+2\mathbb{E}^1\Big\{\sup\limits_{s\leq t}\Big|\int_0^s\int_\mathbb{H}
\frac{\lambda^{\epsilon,w_2}(r,z)-1}{\epsilon}\widetilde{N^1}(dz,dr)\Big|^2\Big\}\cr
\leq&C\mathbb{E}^1\Big\{\int_0^t\int_\mathbb{H}\Big(\sup\limits_{r\leq s}
\Big|\frac{Z^{\epsilon,w_2}_{r-}-1}{\epsilon}\Big|\Big)^2\Big|\lambda^{\epsilon,w_2}(s,z)-1\Big|^2\rho(z)\mu(dz)ds\Big|\Big\}\cr
&+C\mathbb{E}^1\Big\{\int_0^t\int_\mathbb{H}\Big|\frac{\lambda^{\epsilon,w_2}(s,z)-1}{\epsilon}\Big|^2\rho(z)\mu(dz)ds\Big\},
\end{align}
where $C$ is a constant,which may change value from line to line and is independent of $\epsilon$.
Note that
\begin{align}
\label{4}
\Big|\frac{\lambda^\epsilon(s,z)-1}{\epsilon}\Big|^2
=&\frac{\big|\varphi\big(z,\epsilon V(s)\big)\frac{\rho\big(z+\epsilon V(s)\big)}{\rho(z)}-1\big|^2}{\epsilon^2}\cr
\leq&2\frac{\varphi^2\big(z,\epsilon V(s)\big)}{\rho^2(z)}
\Big|\frac{\rho\big(z+\epsilon V(s)\big)-\rho(z)}{\epsilon}\Big|^2\cr
&+2\Big|\frac{\varphi\big(z,\epsilon V(s)\big)-1}{\epsilon}\Big|^2,
\end{align}
and
\begin{align}
\label{5}
\varphi^2\big(z,\epsilon V^{w_2}(s)\big)
=&\exp\Big\{\langle z,2\epsilon Q^{-1}V^{w_2}(s)\rangle-\epsilon^2\langle V^{w_2}(s), Q^{-1}V^{w_2}(s)\rangle\Big\}\cr
=&\varphi\big(z,2\epsilon V^{w_2}(s)\big)\exp\Big\{\epsilon^2\langle V^{w_2}(s), Q^{-1}V^{w_2}(s)\rangle\Big\}.
\end{align}
Therefore, by mean value theorem and ($\ref{5}$), there exists constant a $C_1>0$ independent of $\epsilon$ such that
\begin{align}
\label{47}
I_1
:=&\int_\mathbb{H}\frac{\varphi^2\big(z,\epsilon V^{w_2}(s)\big)}{\rho(z)}
\Big|\frac{\rho\big(z+\epsilon V^{w_2}(s)\big)-\rho(z)}{\epsilon}\Big|^2\mu(dz)\cr
\leq&\frac{1}{\delta}\|\nabla\rho\|_{\infty}^2|V^{w_2}(s)|^2\int_\mathbb{H}\varphi^2\big(z,\epsilon V^{w_2}(s)\big)\mu(dz)\cr
=&\frac{1}{\delta}\|\nabla\rho\|_{\infty}^2|V^{w_2}(s)|^2\exp\Big\{\langle V^{w_2}(s), Q^{-1}V^{w_2}(s)\rangle\Big\}\cr
\leq&C_1,
\end{align}
and
\begin{align}
\label{48}
I_2
:=&\int_\mathbb{H}\Big|\frac{\varphi\big(z,\epsilon V^{w_2}(s)\big)-1}{\epsilon}\Big|^2\rho(z)\mu(dz)\cr
=&\int_\mathbb{H}\varphi^2\big(z,\epsilon_2V^{w_2}(s)\big)\Big|\langle z-\epsilon_2V^{w_2}(s),
Q^{-1}V^{w_2}(s)\rangle\Big|^2\rho(z)\mu(dz)\cr
\leq&2|Q^{-1}V^{w_2}(s)|^2\int_\mathbb{H}\varphi^2\big(z,\epsilon_2V^{w_2}(s)\big)|z|^2\rho(z)\mu(dz)\cr
&+2|Q^{-1}V^{w_2}(s)|^2|V^{w_2}(s)|^2\int_\mathbb{H}\varphi^2\big(z,\epsilon_2V^{w_2}(s)\big)\rho(z)\mu(dz)\cr
:=&I_{21}+I_{22},
\end{align}
where $\epsilon_2\in(0,\epsilon)$ is a proper constant.
Furthermore, there exist constants $\epsilon_3\in(0,\epsilon_2)$, $C_2$ and $C_3$ independent of $\epsilon$ such that
\begin{align}
\label{49}
I_{21}
\leq&2|Q^{-1}V^{w_2}(s)|^2\exp\Big\{\langle V^{w_2}(s),
Q^{-1}V^{w_2}(s)\rangle\Big\}\int_\mathbb{H}|z|^2\rho(z)\mu(dz+2\epsilon_2V^{w_2}(s))\cr
\leq&4|Q^{-1}V^{w_2}(s)|^2\exp\Big\{\langle V^{w_2}(s), Q^{-1}V^{w_2}(s)\rangle\Big\}\cr
&\times\int_\mathbb{H}\Big(|z+\epsilon_2V^{w_2}(s)|^2+|\epsilon_2V^{w_2}(s)|^2\Big)\cr
&~~~~~\times\Big(|\rho(z+\epsilon_2V^{w_2}(s))-\rho(z)|+\rho(z+\epsilon_2V^{w_2}(s))\Big)\mu(dz+\epsilon_2V^{w_2}(s))\cr
\leq&4|Q^{-1}V^{w_2}(s)|^2\exp\Big\{\langle V^{w_2}(s), Q^{-1}V^{w_2}(s)\rangle\Big\}\cr
&\times\Bigg\{\int_\mathbb{H}|z+\epsilon_2V^{w_2}(s)|^2\big|\langle\nabla\rho(z+\epsilon_3V^{w_2}(s)),
\epsilon_2V^{w_2}(s)\rangle\big|\mu(dz+\epsilon_2V^{w_2}(s))\cr
&~~~+\int_\mathbb{H}|z+\epsilon_2V^{w_2}(s)|^2\rho(z+\epsilon_2V^{w_2}(s))\mu(dz+\epsilon_2V^{w_2}(s))\cr
&~~~+|V^{w_2}(s)|^2\int_\mathbb{H}|\langle\nabla\rho(z+\epsilon_3V^{w_2}(s)), \epsilon_2V^{w_2}(s)\rangle|
\mu(dz+\epsilon_2V^{w_2}(s))\cr
&~~~+|V^{w_2}(s)|^2\int_\mathbb{H}\rho(z+\epsilon_2V^{w_2}(s))\mu(dz+\epsilon_2V^{w_2}(s))\Bigg\}\cr
\leq&4|Q^{-1}V^{w_2}(s)|^2\exp\Big\{\langle V^{w_2}(s), Q^{-1}V^{w_2}(s)\rangle\Big\}\cr
&\times\Bigg\{|V^{w_2}(s)|\|\nabla\rho\|_{\infty}\int_\mathbb{H}|z|^2\mu(dz)+\int_\mathbb{H}|z|^2\rho(z)\mu(dz)
+|V^{w_2}(s)|^3\|\nabla\rho\|_{\infty}+\lambda|V^{w_2}(s)|^2\Bigg\}\cr
\leq& C_2,
\end{align}
and
\begin{align}
\label{50}
I_{22}=&2|Q^{-1}V^{w_2}(s)|^2|V^{w_2}(s)|^2\int_\mathbb{H}\varphi^2\big(z,\epsilon_2 V^{w_2}(s)\big)\rho(z)\mu(dz)\cr
\leq&2|Q^{-1}V^{w_2}(s)|^2|V^{w_2}(s)|^2\exp\Big\{\epsilon^2\langle V^{w_2}(s), Q^{-1}V^{w_2}(s)\rangle\Big\}\cr
&\times\int_\mathbb{H}\Big(|\rho(z+\epsilon_2V^{w_2}(s))-\rho(z)|+\rho(z+\epsilon_2V^{w_2}(s))\Big)
\mu(dz+2\epsilon_2V^{w_2}(s))\cr
\leq&2|Q^{-1}V^{w_2}(s)|^2|V^{w_2}(s)|^2\exp\Big\{\epsilon^2\langle V^{w_2}(s), Q^{-1}V^{w_2}(s)\rangle\Big\}
\Big\{\|\nabla\rho\|_{\infty}|V^{w_2}(s)|+\lambda\Big\}\cr
\leq&C_3.
\end{align}
Combining $(\ref{47})$, ($\ref{48}$), ($\ref{49}$), ($\ref{50}$) with ($\ref{4}$), we arrive at
\begin{align}
\label{51}
\int_{\mathbb{H}}\Big|\frac{\lambda^{\epsilon,w_2}(s,z)-1}{\epsilon}\Big|^2\rho(z)\mu(dz)\leq 2(C_1+C_2+C_3).
\end{align}
It follows from ($\ref{1}$) and ($\ref{51}$) that
\begin{align*}
\mathbb{E}^1\Big\{\sup\limits_{s\leq t}\Big|\frac{Z^{\epsilon,w_2}_s-1}{\epsilon}\Big|^2\Big\}
\leq C\mathbb{E}^1\Big\{\int_0^t\Big(\sup\limits_{r\leq s}\Big|\frac{Z^{\epsilon,w_2}_r-1}{\epsilon}\Big|^2ds\Big\}+C.
\end{align*}
Applying Gronwall's inequality, we deduce
\begin{align*}
\mathbb{E}^1\Big\{\sup\limits_{s\leq t}\Big|\frac{Z^{\epsilon,w_2}_s-1}{\epsilon}\Big|^2\Big\}<C(t),
\end{align*}
where $C(t)$ is a constant independent of $\epsilon$.
Consequently, the claim is proved.
\end{proof}

With the help of above two lemmas, we are ready to derive the following integration by parts formula.
\begin{thm}
Suppose $(\bf{H1})$ holds. For $V\in\mathscr{V}_2$ and $f\in C^2_b(\mathbb{H})$,
\begin{align}\label{6}
\mathbb{E}\Big\{D_Vf(L_t)\Big\}=-\mathbb{E}\Big\{f(L_t)M_t\Big\},~~t\leq T,
\end{align}
where $M_t=\int_0^t\int_\mathbb{H}\Big(\langle z, Q^{-1}V(s)\rangle\
+\langle\nabla\log\rho(z), V(s)\rangle\Big)\widetilde{N^1}(dz,ds)$.
\end{thm}
\begin{proof}
We give the proof in three steps.\\
Step 1. Assume $V\in\mathscr{V}_1$ and $\rho\geq\delta$ for some $\delta>0$. By virtue of Lemma 2.1, for any fixed $w_2\in\Omega_2$
and $\epsilon\in(0,1)$,
we have
\begin{align*}
\mathbb{E}^1f(L^{w_2}_t)=\mathbb{E}^1\Big\{f(L^{\epsilon,w_2}_t)Z_t^{\epsilon,w_2}\Big\}.
\end{align*}
Therefore,
\begin{align*}
\mathbb{E}^1\frac{f(L^{\epsilon,w_2}_t)Z_t^{\epsilon,w_2}-f(L^{w_2}_t)}{\epsilon}=0.
\end{align*}
Furthermore,
{\begin{align}
\label{7}
\mathbb{E}^1\frac{f(L^{\epsilon,w_2}_t)-f(L_t)}{\epsilon}
+\mathbb{E}^1\frac{f(L^{\epsilon,w_2}_t)(Z_t^{\epsilon,w_2}-1-\epsilon M^{w_2}_t)}{\epsilon}
+\mathbb{E}^1f(L^{\epsilon,w_2}_t)M^{w_2}_t=0,
\end{align}}
where
\begin{align*}
M^{w_2}_t=\int_0^t\int_\mathbb{H}\Big(\langle z, Q^{-1}V^{w_2}(s)\rangle
+\langle\nabla\log\rho(z), V^{w_2}(s)\rangle\Big)\widetilde{N^1}(dz,ds).
\end{align*}
Taking expectation w.r.t. $\mathbb{P}^2$ in both sides of ($\ref{7}$), we have
{\begin{align}
\label{52}
\mathbb{E}\frac{f(L^{\epsilon}_t)-f(L_t)}{\epsilon}
+\int_{\Omega_2}\mathbb{E}^1\frac{f(L^{\epsilon,w_2}_t)(Z_t^{\epsilon,w_2}-1-\epsilon M^{w_2}_t)}{\epsilon}\mathbb{P}^2(dw_2)
+\mathbb{E}f(L^\epsilon_t)M_t=0.
\end{align}}
The Definition 2.1 implies
\begin{align}
\label{8}
\lim\limits_{\epsilon\rightarrow0}\mathbb{E}\frac{f(L^\epsilon_t)-f(L_t)}{\epsilon}=\mathbb{E}D_Vf(L_t).
\end{align}
Moreover,
\begin{align}
\label{9}
\lim\limits_{\epsilon\rightarrow0}\mathbb{E}\{f(L^\epsilon_t)M_t\}=\mathbb{E}\{f(L_t)M_t\}.
\end{align}
Therefore, it is sufficient to prove
\begin{align*}
\lim\limits_{\epsilon\rightarrow0}\mathbb{E}^1\frac{f(L^{\epsilon,w_2}_t)(Z_t^{\epsilon,w_2}-1-\epsilon M^{w_2}_t)}{\epsilon}=0.
\end{align*}
By $(\ref{3})$ and the fact that $\lambda<\infty$, one has
\begin{align*}
\lim\limits_{\epsilon\rightarrow0}\frac{Z_t^{\epsilon,w_2}-1}{\epsilon}
=&\lim\limits_{\epsilon\rightarrow0}\int_0^t\int_\mathbb{H}\frac{Z_{s-}^{\epsilon,w_2}
(\lambda^{\epsilon,w_2}(s,z)-1)}{\epsilon}\widetilde{N^1}(dz,ds)\\
=&\int_0^t\int_\mathbb{H}\lim\limits_{\epsilon\rightarrow0}\frac{Z_{s-}^{\epsilon,w_2}
(\lambda^{\epsilon,w_2}(s,z)-1)}{\epsilon}\widetilde{N^1}(dz,ds)\\
=&\int_0^t\int_\mathbb{H}\frac{d}{d\epsilon}\Big|_{\epsilon=0}
\Big(\varphi(z,\epsilon V^{w_2}(s))\frac{\rho(z+\epsilon V^{w_2}(s))}{\rho(z)}\Big)\widetilde{N^1}(dz,ds)\\
=&\int_0^t\int_\mathbb{H}\Big(\langle z, Q^{-1}V^{w_2}(s)\rangle+\langle\nabla\log\rho(z), V^{w_2}(s)\rangle\Big)
\widetilde{N^1}(dz,ds).
\end{align*}
Combining this with Lemma 2.2, we derive
\begin{align}
\label{10}\lim\limits_{\epsilon\rightarrow0}\mathbb{E}^1\frac{|Z_t^{\epsilon,w_2}-1-\epsilon M^{w_2}_t|}{\epsilon}=0.
\end{align}
By the dominated convergence theorem and ($\ref{52}$)-($\ref{10}$), we obtain ($\ref{6}$).

Step 2. Assume $V\in\mathscr{V}_1$.
For each $n\in\mathbb{N}$, let $\mathbb{P}^{1,n}$ be a probability measure on path space $W_{1,n}$ such that the coordinate
process $L^{1,n}_t=w_{1,n,}(t)$ is a purely jump L\'evy process with characteristic measure $\frac{1}{n}\mu(dz)$.
The associated jump measure is denoted by $N^{1,n}$.
Define
\begin{align*}
\hat{\Omega}=\prod_{n=1}^\infty W_{1,n}\times\Omega,~~\hat{\mathbb{P}}=\prod_{n=1}^\infty \mathbb{P}^{1,n}\times\mathbb{P},
\end{align*}
and
\begin{align*}
L^n_t=L^{1,n}_t+L_t.
\end{align*}
Then the jump measure and characteristic measure of $L^1_t+L^{1,n}_t$ are $N_n(dz,ds):=N^1(dz,ds)+N^{1,n}(dz,ds)$ and
$(\rho(z)+\frac{1}{n})\mu(dz)$ respectively.
By Step 1, for $V\in\mathscr{V}_1$
\begin{align}
\label{17}
\hat{\mathbb{E}}D_Vf(L^n_t)=-\hat{\mathbb{E}}\{f(L^n_t)M^n_t\},
\end{align}
where
\begin{align*}
M^n_t=\int_0^t\int_\mathbb{H}\Big(\langle z, Q^{-1}V(s)\rangle+\langle\frac{\nabla\rho(z)}{\rho(z)+\frac{1}{n}}, V(s)\rangle\Big)
\widetilde{N_n}(dz,ds).
\end{align*}
Note that,
\begin{align*}
&|\hat{\mathbb{E}}D_Vf(L^n_t)-\mathbb{E}D_Vf(L_t)|\cr
=&|\hat{\mathbb{E}}D_Vf(L^n_t)-\hat{\mathbb{E}}D_Vf(L_t)|\cr
=&\Big|\hat{\mathbb{E}}\langle\nabla f(L^n_t),\int_0^t\int_\mathbb{H}V(s)N_n(dz,ds)\rangle
-\hat{\mathbb{E}}\langle\nabla f(L_t),\int_0^t\int_\mathbb{H}V(s)N^1(dz,ds)\rangle\Big|\cr
\leq&\Big|\hat{\mathbb{E}}\langle\nabla f(L^n_t),\int_0^t\int_\mathbb{H}V(s)N_n(dz,ds)\rangle
-\hat{\mathbb{E}}\langle\nabla f(L^n_t),\int_0^t\int_\mathbb{H}V(s)N^1(dz,ds)\rangle\Big|\cr
&+\Big|\hat{\mathbb{E}}\langle\nabla f(L^n_t),\int_0^t\int_\mathbb{H}V(s)N^1(dz,ds)\rangle
-\hat{\mathbb{E}}\langle\nabla f(L_t),\int_0^t\int_\mathbb{H}V(s)N^1(dz,ds)\rangle\Big|\cr
\leq&\|\nabla f\|_\infty\hat{\mathbb{E}}\int_0^t\int_\mathbb{H}|V(s)|N^{1,n}(dz,ds)\cr
&+\|\nabla^2f\|_\infty\mathbb{E}\Big\{\int_0^t\int_\mathbb{H}|z|N^{1,n}(dz,ds)\int_0^t\int_\mathbb{H}|V(s)|N^1(dz,ds)\Big\}\cr
\leq&\frac{1}{n}\Big\{\|\nabla f\|_\infty\|V\|_\infty t+\|\nabla^2f\|_\infty\|V\|_\infty\lambda\int_\mathbb{H}|z|\mu(dz)t^2\Big\}
\rightarrow0,~~~as~~n\rightarrow\infty,
\end{align*}
and
\begin{align*}
|\hat{\mathbb{E}}\{f(L^n_t)M^n_t\}-\mathbb{E}\{f(L_t)M_t\}|=&|\hat{\mathbb{E}}\{f(L^n_t)M^n_t\}-\hat{\mathbb{E}}\{f(L_t)M_t\}|\cr
\leq&|\hat{\mathbb{E}}\{f(L^n_t)M^n_t\}-\hat{\mathbb{E}}\{f(L^n_t)M_t\}|\cr
&+|\hat{\mathbb{E}}\{f(L^n_t)M_t\}-\hat{\mathbb{E}}\{f(L_t)M_t\}|\cr
:=&I_1+I_2.
\end{align*}
As for $I_1$, we have
\begin{align*}
I_1\leq&\|f\|_\infty\hat{\mathbb{E}}|M^n_t-M_t|\cr
=&\|f\|_\infty\hat{\mathbb{E}}\Big|\int_0^t\int_\mathbb{H}\Big(\langle z, Q^{-1}V(s)\rangle
+\langle\frac{\nabla\rho(z)}{\rho(z)+\frac{1}{n}}, V(s)\rangle\Big)\widetilde{N_n}(dz,ds)\cr
&~~~~~~~~~~~-\int_0^t\int_\mathbb{H}\Big(\langle z, Q^{-1}V(s)\rangle
+\langle\frac{\nabla\rho(z)}{\rho(z)}, V(s)\rangle\Big)\widetilde{N^1}(dz,ds)\Big|\cr
\leq&\|f\|_\infty\Bigg\{\hat{\mathbb{E}}\Big|\int_0^t\int_\mathbb{H}\langle z, Q^{-1}V(s)\rangle\widetilde{N^{1,n}}(dz,ds)\Big|\cr
&~~~~~~~~+\hat{\mathbb{E}}\Big|\int_0^t\int_\mathbb{H}\langle\frac{\nabla\rho(z)}{\rho(z)+\frac{1}{n}}-\frac{\nabla\rho(z)}{\rho(z)},
V(s)\rangle\widetilde{N^1}(dz,ds)\Big|\cr
&~~~~~~~~+\hat{\mathbb{E}}\Big|\int_0^t\int_\mathbb{H}\langle\frac{\nabla\rho(z)}{\rho(z)+\frac{1}{n}},V(s)\rangle\widetilde{N^{1,n}}(dz,ds)\Bigg\}\cr
\leq&2\|f\|_\infty\Bigg\{\frac{t}{n}\|Q^{-1}V\|_\infty\int_\mathbb{H}|z|\mu(dz)
+t\|V\|_\infty\int_\mathbb{H}\Big|\frac{\rho(z)}{\rho(z)+\frac{1}{n}}-1\Big||\nabla\rho(z)|\mu(dz)\cr
&~~~~~~~~~~+t\int_\mathbb{H}\frac{|\nabla\rho(z)|}{n\rho(z)+1}\mu(dz)\Bigg\}\rightarrow0,~~as~~n\rightarrow\infty.
\end{align*}
Meanwhile,
\begin{align*}
I_2=&|\hat{\mathbb{E}}\{f(L^n_t)M_t\}-\hat{\mathbb{E}}\{f(L_t)M_t\}|\cr
\leq&\|\nabla f\|_\infty\hat{\mathbb{E}}|L^{1,n}_t||M_t|\cr
\leq&\|\nabla f\|_\infty\hat{\mathbb{E}}|L^{1,n}_t|\hat{\mathbb{E}}|M_t|\cr
\leq&\|\nabla f\|_\infty\frac{t}{n}\int_\mathbb{H}|z|\mu(dz)\hat{\mathbb{E}}|M_t|\rightarrow0,~~as~~n\rightarrow\infty.
\end{align*}
Considering above estimates and letting $n\rightarrow\infty$ in ($\ref{17}$), we get ($\ref{6}$) for $V\in\mathscr{V}_1$.

Step 3. Assume $V\in\mathscr{V}_2$. For $n\in\mathbb{N}$, define
\begin{align*}
V_n(t)=V(t)I_{[0,n]}(|Q^{-1}V(t)|),~~t\geq0.
\end{align*}
By Step 2, one arrives at
\begin{align}
\label{12}
\mathbb{E}D_{V_n}f(L_t)=-\mathbb{E}\{f(L_t)M^n_t\},
\end{align}
for $M^n_t=\int_0^t\int_\mathbb{H}\Big(\langle z, Q^{-1}V_n(s)\rangle+\langle\nabla\log\rho(z), V_n(s)\rangle\Big)\widetilde{N^1}(dz,ds)$.
It is easy to check that
\begin{align*}
\mathbb{E}|D_{V_n}f(L_t)-D_Vf(L_t)|\rightarrow0\text{ and }\mathbb{E}|M^n_t-M_t|\rightarrow0,~~n\rightarrow\infty,
\end{align*}
Therefore, let $n\rightarrow\infty$ in $(\ref{12})$ and we finish the proof.
\end{proof}

\section{Proofs of Main Results.}
In this section, we would like to give the proofs of main results. Denote
\begin{align*}
C_b^2(\mathbb{H}\rightarrow\mathbb{H})=&\Big\{G:\mathbb{H}\rightarrow\mathbb{H}\big|
G\text{ is bounded, differential, with bounded} \\&\text{ ~~~~~~~~~~~~~~~~~~and continuous derivatives up to order }2.\Big\}.
\end{align*}
Before we move on, it is necessary for us to prove the existence of $L^1$-derivative of ($\ref{2}$).
\begin{prp}
Assume $A$ generates a $C_0$-semigroup $\{S(t)\}_{t\geq0}$ and $F\in C_b^2(\mathbb{H\rightarrow\mathbb{H}})$.
If a predictable process $V$ satisfies $\mathbb{E}\int_0^T|V(s)|^2ds<\infty$,
then $X_t$ has an $L^1$-derivative in direction $V$. Moreover, the $L^1$-derivative satisfies
\begin{equation}\label{18}
\begin{cases}
dD_VX_t=AD_VX_tdt+\nabla F(X_t)D_VX_tdt+\int_{\mathbb{H}}V(s)N^1(dz,dt)\\
D_VX_0=0.
\end{cases}
\end{equation}
\end{prp}
\begin{proof}
By classical results of SPDEs, the solution of Eq. ($\ref{18}$) admits a unique solution
\begin{align}
\label{13}
D_VX_t=\int_0^tS(t-s)\nabla F(X_s)D_VX_sds+\int_0^t\int_\mathbb{H}V(s)N^1(dz,ds).
\end{align}
Now we aim to prove $D_VX_t$ is the $L^1$-derivative of $X_t$. It is easy to check the integrability of $D_VX_t$.
We shall prove
\begin{align}
\label{16}
E\Big\{\sup\limits_{s\leq t}\Big|\frac{X_s^{\epsilon}-X_s}{\epsilon}-D_VX_s\Big|\Big\}\rightarrow0, ~~\epsilon\rightarrow0.
\end{align}
In fact,
\begin{align*}
X_t^{\epsilon}
=&x+\int_0^tS(t-s)F(X_s^{\epsilon})ds+\int_0^tS(t-s)dL^\epsilon_s\\
=&x+\int_0^tS(t-s)F(X_s^{\epsilon})ds+\int_0^tS(t-s)dL_s+\epsilon\int_0^t\int_\mathbb{H}S(t-s)V(s)N^1(dz,ds),
\end{align*}
then
\begin{align}
\label{14}
X_t^\epsilon-X_t=\int_0^tS(t-s)(F(X_s^\epsilon)-F(X_s))ds+\epsilon\int_0^t\int_{\mathbb{H}}S(t-s)V(s)N^1(dz,ds).
\end{align}
Therefore,
\begin{align*}
|X_t^\epsilon-X_t|&\leq\int_0^t|S(t-s)(F(X_s^\epsilon)-F(X_s))|ds+\epsilon\int_0^t\int_{\mathbb{H}}|S(t-s)V(s)|N^1(dz,ds)\\
&\leq C\|\nabla F\|_{\infty}e^{\delta t}\int_0^te^{-\delta s}|X_s^\epsilon-X_s|ds
+\epsilon Ce^{\delta t}\int_0^t\int_{\mathbb{H}}e^{-\delta s}|V(s)|N^1(dz,ds).
\end{align*}
Furthermore,
\begin{align*}
&\sup\limits_{s\leq t}\Big\{e^{-\delta s}|X_s^\epsilon-X_s|\Big\}\\
&\leq C\|\nabla F\|_{\infty}\int_0^t\sup\limits_{r\leq s}\Big\{e^{-\delta r}|X_r^\epsilon-X_r|\Big\}ds
+\epsilon C\int_0^t\int_{\mathbb{H}}e^{-\delta s}|V(s)|N^1(dz,ds).
\end{align*}
Using Gronwall's inequality, one obtains
\begin{align*}
\sup\limits_{s\leq t}\Big\{e^{-\delta s}|X_s^\epsilon-X_s|\Big\}
\leq\epsilon C\exp\{C\|\nabla F\|_{\infty}t\}\int_0^t\int_{\mathbb{H}}e^{-\delta s}|V(s)|N^1(dz,ds),
\end{align*}
which yields
\begin{align}
\label{15}
\sup\limits_{s\leq t}|X_s^\epsilon-X_s|\leq\epsilon C\exp\{\delta t+C\|\nabla F\|_{\infty}t\}\int_0^t\int_{\mathbb{H}}e^{-\delta s}|V(s)|N^1(dz,ds).
\end{align}
From $(\ref{13})$, $(\ref{14})$ and Taylor's formula, it follows  that
\begin{align*}
&\Big|\frac{X_t^\epsilon-X_t}{\epsilon}-D_VX_t\Big|\\
\leq&C\int_0^te^{\delta(t-s)}\Big|\frac{F(X_s^\epsilon)-F(X_s)}{\epsilon}-\nabla F(X_s)D_VX_s\Big|ds\\
\leq&C\int_0^te^{\delta(t-s)}\Big\{\|\nabla F\|_{\infty}\Big|\frac{X_s^\epsilon-X_s}{\epsilon}-D_VX_s\Big|
+\|\nabla^2F\|_{\infty}\frac{|X_s^\epsilon-X_s|^2}{\epsilon}\Big\}ds.
\end{align*}
By the similar argument above, we have
\begin{align*}
&\sup\limits_{s\leq t}\Big|\frac{X_s^\epsilon-X_s}{\epsilon}-D_VX_s\Big|\\
\leq&C\|\nabla^2F\|_\infty t\exp\{C\|\nabla F\|_\infty t+\delta t\}\sup\limits_{s\leq t}\frac{|X_s^\epsilon-X_s|^2}{\epsilon}.
\end{align*}
Combining it with ($\ref{15}$), one arrives at
\begin{align*}
\mathbb{E}\Big\{\sup\limits_{s\leq t}\Big|\frac{X_s^\epsilon-X_s}{\epsilon}-D_VX_s\Big|\Big\}
\leq&\hat{C}_1\epsilon\mathbb{E}\Big\{\int_0^t\int_{\mathbb{H}}|V(s)|N^1(dz,ds)\Big\}^2\\
\leq&\hat{C}_2\epsilon.
\end{align*}
where $\hat{C}_1$ and $\hat{C}_2$ are constants independent of $\epsilon$.
\end{proof}
Let $J_t=\nabla_xX^x_t$. For $s\leq t$, let $J_{st}$ be the solution of following equation:
\begin{equation}\label{19}
\begin{cases}
dJ_{st}=AJ_{st}dt+\nabla F(X_t)J_{st}dt,\\
J_{ss}=I.
\end{cases}
\end{equation}
Then it can be proved that
\begin{align*}
J_t=J_{0t}=J_{st}J_s
\end{align*}
By ($\ref{13}$) and ($\ref{19}$), we can deduce
\begin{align}
\label{20}
D_VX_t=\int_0^t\int_{\mathbb{H}}J_{st}V(s)N(dz,ds).
\end{align}

\noindent{\bf{Proof of Theorem 1.1.}}We show the proof in two steps.\\
Step1: Assume $F\in C_b^2(\mathbb{H}\rightarrow\mathbb{H})$. By ($\ref{19}$), one has
\begin{align*}
J_t=S(t)+\int_0^tS(t-s)\nabla F(X_s)J_sds.
\end{align*}
Then
\begin{align*}
\|J_t\|\leq e^{-\gamma_1t}+\|\nabla F\|_{\infty}\int_0^t\|J_s\|e^{-\gamma_1(t-s)}ds.
\end{align*}
Gronwall's inequality implies
\begin{align}
\label{24}
\|J_t\|\leq\exp\{(-\gamma_1+\|\nabla F\|_{\infty})t\}.
\end{align}
Observe that
\begin{align}
\label{25}
Q^{-1}J_t&=Q^{-1}S(t)+Q^{-1}\int_0^tS(t-s)\nabla F(X_s)J_sds\cr
&=Q^{-1}S(t)+\lim\limits_{\epsilon\rightarrow0}(Q+\epsilon I)^{-1}\int_0^tS(t-s)\nabla F(X_s)J_sds\cr
&=Q^{-1}S(t)+\lim\limits_{\epsilon\rightarrow0}\int_0^t(Q+\epsilon I)^{-1}S(t-s)\nabla F(X_s)J_sds\cr
&=Q^{-1}S(t)+\int_0^t\lim\limits_{\epsilon\rightarrow0}(Q+\epsilon I)^{-1}S(t-s)\nabla F(X_s)J_sds\cr
&=Q^{-1}S(t)+\int_0^tQ^{-1}S(t-s)\nabla F(X_s)J_sds,
\end{align}
where in the forth equality we use the dominated convergence theorem. By ($\ref{24}$) and ($\ref{25}$), we get
\begin{align*}
\int_0^t\|Q^{-1}J_s\|ds
&\leq\int_0^t\|Q^{-1}S(s)\|ds+\int_0^t\int_0^s\|Q^{-1}S(s-r)\|\|\nabla F(X_r)\|\|J_r\|drds\\
&\leq\Big(1+t\|\nabla F\|_{\infty}\Big)\int_0^t\|Q^{-1}S(s)\|ds<\infty.
\end{align*}
Taking $V(s)=J_s\xi$ in ($\ref{20}$), we obtain $D_VX_t= N^1_tJ_t\xi$, and
\begin{align}
\label{21}
\frac{I_{[N^1_t\geq1]}}{N^1_t}D_VX_t=J_tI_{[N^1_t\geq1]}\xi.
\end{align}
Since $D_VN^1_t=0$, then
\begin{align}
\label{22}
D_V\big\{\frac{I_{[N^1_t\geq1]}}{N^1_t}\big\}=0.
\end{align}
It follows from ($\ref{21}$),  ($\ref{22}$) and ($\ref{6}$) that
\begin{align}
\label{23}
&\nabla_\xi P_t^1f(x)\cr
=&\nabla_\xi\mathbb{E}\Big\{f(X_t^x)I_{[N^1_t\geq1]}\Big\}\cr
=&\mathbb{E}\langle\nabla f(X_t^x),J_t\xi I_{[N^1_t\geq1]}\rangle\cr
=&\mathbb{E}\langle\nabla f(X_t^x),\frac{I_{[N^1_t\geq1]}}{N^1_t}D_VX^x_t\rangle\cr
=&\mathbb{E}\Big\{D_Vf(X^x_t)\frac{I_{[N^1_t\geq1]}}{N^1_t}\Big\}\cr
=&\mathbb{E}\Big\{D_V\big(f(X^x_t)\frac{I_{[N^1_t\geq1]}}{N^1_t}\big)\Big\}\cr
=&-\mathbb{E}\Big\{f(X_t^x)\frac{I_{[N^1_t\geq1]}}{N^1_t}
\int_0^t\int_\mathbb{H}\Big(\langle z, Q^{-1}J_s\xi\rangle+\langle\nabla\log\rho(z), J_s\xi\rangle\Big)\widetilde{N^1}(dz,ds)\Big\}.
\end{align}
Step2: Assume $F\in C^1_b(\mathbb{H}\rightarrow\mathbb{H})$ and $\nabla F$ is Lipschitz continuous.
We aim to construct approximation sequence $\{F_k\}_{k\geq1}\subset C^2_b(\mathbb{H}\rightarrow\mathbb{H})$ such that
$F_k\rightarrow F$ and $\nabla F_k\rightarrow\nabla F$ in pointwise sense as $k\rightarrow\infty$.
For $k\geq1$, we take a sequence of non-negative, twice differential function $\{g_k\}_{k\geq1}$ such that
\begin{align*}
{\rm{Supp}}\{g_k\}\subset\{y\in\mathbb{R}^k:|y|_{\mathbb{R}^k}\leq\frac{1}{k}\}
\end{align*}
and
\begin{align*}
\int_{\mathbb{R}^k}g_k(y)dy=1.
\end{align*}
Identifying $\mathbb{R}^k$ with span$\{e_1,\cdots,e_k\}$, we define
\begin{align}
\label{43}
F_k(x)=\int_{\mathbb{R}^k}g_k(y-\Pi_kx)F\Big(\sum\limits_{i=1}^ky_ie_i\Big)dy,
\end{align}
then $F_k$ is a twice differentiable function with  bounded and continuous derivatives. Moreover,
\begin{align}
\nabla F_k(x)=\int_{\mathbb{R}^k}g_k(y)\nabla F\Big(\sum\limits_{i=1}^ky_ie_i+\Pi_kx\Big)\Pi_kdy.
\end{align}
For any $x,\widetilde{x}\in\mathbb{H}$,
\begin{align*}
&\|\nabla F_k(x)-\nabla F_k(\widetilde{x})\|\cr
=&\Big\|\int_{\mathbb{R}^k}g_k(y)\Big[\nabla F\Big(\sum\limits_{i=1}^ky_ie_i+\Pi_kx\Big)
-\nabla F\Big(\sum\limits_{i=1}^ky_ie_i+\Pi_k\widetilde{x}\Big)\Big]\Pi_kdy\Big\|\cr
\leq&\int_{\mathbb{R}^k}g_k(y)\|\nabla F\|_{Lip}|x-\widetilde{x}|dy\cr
\leq&\|\nabla F\|_{Lip}|x-\widetilde{x}|,
\end{align*}
where $\|\nabla F\|_{Lip}$ denotes the smallest Lipschitz constant.
This implies
\begin{align}
\label{37}
\sup\limits_{k\geq1}\|\nabla^2F_k\|_{\infty}\leq\|\nabla F\|_{Lip}.
\end{align}
Consider the following equation for any $k\geq1$,
\begin{equation}
\label{34}
\begin{cases}
dX^k_t=AX^k_tdt+F_k(X^k_t)dt+dL_t,\\
X^k_0=x.
\end{cases}
\end{equation}
Let $\{X^k_t\}_{t\geq0}$ be the solution of Eq.(\ref{34}) and $\{J^k_t\}_{t\geq0}$ be its derivative w.r.t. the initial value.
Then $\{J^k_t\}_{t\geq0}$ satisfies
\begin{equation}
\label{35}
\begin{cases}
dJ^k_t=AJ^k_tdt+\nabla F_k(X^k_t)J^k_tdt,\\
J^k_0=I.
\end{cases}
\end{equation}
Since
\begin{align*}
|X_t^k-X_t|=&\Big|\int_0^tS(t-s)\big(F_k(X^k_s)-F(X_s)\big)ds\Big|\cr
\leq&\Big|\int_0^tS(t-s)\big(F_k(X^k_s)-F_k(X_s)\big)ds\Big|+\Big|\int_0^tS(t-s)\big(F_k(X_s)-F(X_s)\big)ds\Big|\cr
\leq&\int_0^te^{-\gamma_1(t-s)}\sup\limits_{k\geq1}\|\nabla F_k\|_{\infty}|X^k_s-X_s|ds+\int_0^te^{-\gamma_1(t-s)}\big|F_k(X_s)-F(X_s)\big|ds,
\end{align*}
then with the help of Gronwall's inequality and dominated convergence theorem, we obtain
\begin{align*}
&\lim\limits_{k\rightarrow\infty}|X_t^k-X_t|\cr
\leq&\exp\Big\{-(\gamma_1+\sup\limits_{k\geq1}\|\nabla F_k\|_{\infty})t\Big\}\lim\limits_{k\rightarrow\infty}
\int_0^te^{-\gamma_1(t-s)}\big|F_k(X_s)-F(X_s)\big|ds=0.
\end{align*}
Based on the above estimates, by ($\ref{24}$),we deduce
\begin{align}
\|J^k_t-J_t\|\leq&\int_0^t\|S(t-s)\|\|\nabla F_k(X_s^k)J_s^k-\nabla F(X_s)J_s\|ds\cr
\leq&\int_0^t\|S(t-s)\|\|\nabla F_k(X^k_s)J_s^k-\nabla F_k(X^k_s)J_s\|ds\cr
&+\int_0^t\|S(t-s)\|\|\nabla F_k(X^k_s)J_s-\nabla F_k(X_s)J_s\|ds\cr
&+\int_0^t\|S(t-s)\|\|\nabla F_k(X_s)J_s-\nabla F(X_s)J_s\|ds\cr
\leq&\sup\limits_{k\geq1}\|\nabla F_k\|_{\infty}\int_0^te^{-\gamma_1(t-s)}\|J^k_s-J_s\|ds\cr
&+\sup\limits_{k\geq1}\|\nabla^2F_k\|_{\infty}\int_0^te^{-\gamma_1t+s\|\nabla F\|_{\infty}}|X^k_s-X_s|ds\cr
&+\int_0^te^{-\gamma_1t+s\|\nabla F\|_{\infty}}\|\nabla F_k(X_s)-\nabla F(X_s)\|ds\cr
:=&I_1(k)+I_2(k)+I_3(k),
\end{align}
which implies
\begin{align*}
\|J^k_t-J_t\|\leq\exp\{-\gamma_1t+\sup\limits_{k\geq1}\|\nabla F_k\|_{\infty}t\}\Big(I_2(k)+I_3(k))\Big)\rightarrow0,~~k\rightarrow\infty.
\end{align*}
Define
\begin{align*}
P^{k,1}_tf(x)=\mathbb{E}\Big\{f(X^k_t)I_{[N^1_t\geq1]}\Big\}.
\end{align*}
By ($\ref{23}$), we have
\begin{align}
\label{38}
&\nabla P^{k,1}_tf(x)\cr
=&-\mathbb{E}\Big\{f(X_t^x)\frac{I_{[N^1_t\geq1]}}{N^1_t}
\int_0^t\int_\mathbb{H}\Big(\langle z, Q^{-1}J^k_s\xi\rangle+\langle\nabla\log\rho(z), J^k_s\xi\rangle\Big)\widetilde{N^1}(dz,ds)\Big\}.
\end{align}
Note that as $k\rightarrow\infty$,
\begin{align*}
|\nabla P^{k,1}_tf(x)-\nabla P^1_tf(x)|\leq\|\nabla f\|_{\infty}\mathbb{E}\|J^k_t-J_t\|\rightarrow0,
\end{align*}
and
\begin{align*}
\mathbb{E}\Big|\int_0^t\int_\mathbb{H}\Big(\langle z, Q^{-1}\big(J^k_s-J_s\big)\xi\rangle+\langle\nabla\log\rho(z),
\big(J^k_s-J_s\big)\xi\rangle\Big)\widetilde{N^1}(dz,ds)\Big|\rightarrow0.
\end{align*}
We finish the proof by letting $k\rightarrow\infty$ in ($\ref{38}$).

\noindent{\bf{Proof of Theorem 1.2.}}\\
Step1: Assume $F\in C^2_b(\mathbb{H}\rightarrow\mathbb{H})$.
By Theorem 1.1,triangle inequality and H\"{o}lder inequality, one arrives at
\begin{align}
\label{26}
&|\nabla_\xi P_t^1f(x)|\cr
=&\Big|-\mathbb{E}\Big\{f(X_t^x)\frac{I_{[N^1_t\geq1]}}{N^1_t}\int_0^t\int_\mathbb{H}\Big(\langle z, Q^{-1}J_s\xi\rangle
+\langle\nabla\log\rho(z), J_s\xi\rangle\Big)\widetilde{N^1}(dz,ds)\Big\}\Big|\cr
\leq&\|f\|_{\infty}|\xi|\Bigg\{\Big\{\int_\mathbb{H}|z|^2\rho(z)\mu(dz)\mathbb{E}\frac{I_{[N^1_t\geq1]}}
{(N^1_t)^2}\int_0^t\mathbb{E}\|Q^{-1}J_s\|^2ds\Big\}^{\frac{1}{2}}\cr
&~~~~~~~~~+2\mathbb{E}\int_0^t\int_\mathbb{H}|\nabla\log\rho(z)|\|J_s\|\rho(z)\mu(dz)ds\Bigg\}.
\end{align}
Note that
\begin{align}
\label{27}
\mathbb{E}\frac{I_{[N_t\geq1]}}{(N^1_t)^2}
=&e^{-\lambda t}\sum_{n=1}^\infty\frac{(\lambda t)^n}{n^2n!}\cr
=&\frac{e^{-\lambda t}}{(\lambda t)^2}\sum_{n=1}^\infty\frac{(n+2)(n+1)}{n^2}\frac{(\lambda t)^{n+2}}{(n+2)!}\cr
\leq&6\frac{e^{-\lambda t}}{(\lambda t)^2}\sum_{n=1}^\infty\frac{(\lambda t)^{n+2}}{(n+2)!}\cr
\leq&\frac{6}{(\lambda t)^2},
\end{align}
and by $(\ref{24})$, ($\ref{25}$) and H\"{o}lder inequality, we have
\begin{align}
\label{28}
&\int_0^t\mathbb{E}\|Q^{-1}J_s\|^2ds\cr
=&\int_0^t\mathbb{E}\Big\|Q^{-1}S(s)+\int_0^sQ^{-1}S(s-r)\nabla F(X_r)J_rdr\Big\|^2ds\cr
\leq&2\int_0^t\|Q^{-1}S(s)\|^2ds+2\int_0^t\mathbb{E}\Big\|\int_0^sQ^{-1}S(s-r)\nabla F(X_r)J_rdr\Big\|^2ds\cr
\leq&2\|\nabla F\|_{\infty}^2\int_0^t\Big\{s\int_0^s\|Q^{-1}S(s-r)\|^2\exp\{-2(\gamma_1-\|\nabla F\|_{\infty})r\}dr\Big\}ds\cr
&~+2\int_0^t\|Q^{-1}S(s)\|^2ds\cr
=&2\|\nabla F\|_{\infty}^2\Gamma_t+2\int_0^t\|Q^{-1}S(s)\|^2ds.
\end{align}
where
\begin{align*}
\Gamma_t=\int_0^t\Big\{s\int_0^s\|Q^{-1}S(s-r)\|^2\exp\{-2(\gamma_1-\|\nabla F\|_{\infty})r\}dr\Big\}ds.
\end{align*}
In addition,
\begin{align}
\label{29}
&\mathbb{E}\int_0^t\int_\mathbb{H}|\nabla\log\rho(z)|\|J_s\|\rho(z)\mu(dz)ds\cr
\leq&\int_\mathbb{H}|\nabla\rho(z)|\mu(dz)\int_0^t\exp\{-(\gamma_1-\|\nabla F\|_{\infty})s\}ds\cr
\leq&\frac{\int_\mathbb{H}|\nabla\rho(z)|\mu(dz)}{\gamma_1-\|\nabla F\|_{\infty}}.
\end{align}
With the help of ($\ref{26}$)-($\ref{29}$), we can obtain
\begin{align*}
|\nabla_\xi P_t^1f(x)|
\leq&2\|f\|_{\infty}|\xi|\Bigg\{\sqrt{6}\Big\{\frac{\int_\mathbb{H}|z|^2\rho(z)\mu(dz)}{\lambda^2}
\frac{\|\nabla F\|_{\infty}^2\Gamma_t+\int_0^t\|Q^{-1}S(s)\|^2ds}{t^2}\Big\}^{\frac{1}{2}}\cr
&~~~~~~~~~~~~~~~~~~~~~~~~+\frac{\int_\mathbb{H}|\nabla\rho(z)|\mu(dz)}{\gamma_1-\|\nabla F\|_{\infty}}\Bigg\}.
\end{align*}
Furthermore,
\begin{align}
\label{36}
\|\nabla P_t^1f\|_{\infty}:=&\sup_{|\xi|\leq1}\sup_{x\in\mathbb{H}}|\nabla_\xi P_t^1f(x)|\cr
\leq&2\|f\|_{\infty}\Bigg\{\sqrt{6}\Big\{\frac{\int_\mathbb{H}|z|^2\rho(z)\mu(dz)}{\lambda^2}
\frac{\|\nabla F\|_{\infty}^2\Gamma_t+\int_0^t\|Q^{-1}S(s)\|^2ds}{t^2}\Big\}^{\frac{1}{2}}\cr
&~~~~~~~~~~~~~~~~~~~~~~~~+\frac{\int_\mathbb{H}|\nabla\rho(z)|\mu(dz)}{\gamma_1-\|\nabla F\|_{\infty}}\Bigg\}.
\end{align}
Since $\lim\limits_{t\rightarrow\infty}\frac{\int_0^t\|Q^{-1}S(s)\|^2ds}{t}<\infty$, then
$\lim\limits_{t\rightarrow\infty}\|Q^{-1}S(t)\|^2<\infty$.
Moreover,
\begin{align*}
\lim\limits_{t\rightarrow\infty}\frac{\Gamma_t}{(\lambda t)^2}
=&\lim\limits_{t\rightarrow\infty}\frac{\int_0^t\Big\{s\int_0^s\|Q^{-1}S(s-r)\|^2
\exp\{-2(\gamma_1-\|\nabla F\|_{\infty})r\}dr\Big\}ds}{t^2}\cr
\leq&\lim\limits_{t\rightarrow\infty}\frac{\int_0^t\|Q^{-1}S(t-s)\|^2\exp\{-2(\gamma_1-\|\nabla F\|_{\infty})s\}ds}{2}\cr
=&\lim\limits_{t\rightarrow\infty}\frac{\int_0^t\|Q^{-1}S(s)\|^2\exp\{2(\gamma_1-\|\nabla F\|_{\infty})s\}ds}
{2\exp\{2(\gamma_1-\|\nabla F\|_{\infty})t\}}\cr
=&\lim\limits_{t\rightarrow\infty}\frac{\|Q^{-1}S(t)\|^2}{4(\gamma_1-\|\nabla F\|_{\infty})}<\infty.
\end{align*}
So there exists a constant $C_1$ independent of $t$ and $\lambda$, such that
\begin{align*}
\Big(\sup\limits_{t\geq1}\frac{\Gamma_t}{t^2}\Big)\vee\Big(\sup\limits_{t\geq1}\frac{\int_0^t\|Q^{-1}S(s)\|^2ds}{t^2}\Big)\leq C_1.
\end{align*}
For $t\geq1$, it follows from
($\ref{36}$) that
\begin{align}
\label{30}
&\|\nabla P_t^1f\|_{\infty}\cr
\leq&2\|f\|_{\infty}\Bigg\{\sqrt{6C_1}\Big\{(\|\nabla F\|_{\infty}^2+1)\frac{\int_\mathbb{H}|z|^2\rho(z)\mu(dz)}{\lambda^2}\Big\}^{\frac{1}{2}}
+\frac{\int_\mathbb{H}|\nabla\rho(z)|\mu(dz)}{\gamma_1-\|\nabla F\|_{\infty}}\Bigg\}.
\end{align}
Therefore,
\begin{align}
\label{31}
&|P_tf(x)-P_tf(y)|\cr
=&|P^1_tf(x)-P^1_tf(y)|+|\mathbb{E}\big\{f(X_t^x)I_{[N_t=0]}\big\}-\mathbb{E}\big\{f(X_t^y)I_{[N_t=0]}\big\}|\cr
\leq&\|\nabla P_t^1f\|_{\infty}|x-y|+2\|f\|_{\infty}\mathbb{P}(N^1_t=0)\cr
\leq& 2\|f\|_{\infty}\Bigg\{\Big\{\sqrt{6C_1}\Big((\|\nabla F\|_{\infty}^2+1)\frac{\int_\mathbb{H}|z|^2\rho(z)\mu(dz)}{\lambda^2}\Big)^{\frac{1}{2}}
+\frac{\int_\mathbb{H}|\nabla\rho(z)|\mu(dz)}{\gamma_1-\|\nabla F\|_{\infty}}\Big\}|x-y|\cr
&~~~~~~~~~~~~~~~~~~+e^{-\lambda t}\Bigg\}.
\end{align}
Step2: Assume $(\bf{H4})$ hold. Making use of ($\ref{43}$), we can construct $\{F_n\}_{n\geq1}\subset C^2_b(\mathbb{H}\rightarrow\mathbb{H})$
such that $F_n\rightarrow F$ as $n\rightarrow\infty$ in pointwise sense and $\sup\limits_{n\geq1}\|\nabla F_n\|_{\infty}\leq \|F\|_{Lip}$.
It follows from ($\ref{31}$) that
\begin{align}
\label{45}
&|P^n_tf(x)-P^n_tf(y)|\cr
\leq&2\|f\|_{\infty}\Bigg\{\Big\{\sqrt{6C_1}\Big((\|\nabla F_n\|_{\infty}^2+1)\frac{\int_\mathbb{H}|z|^2\rho(z)\mu(dz)}{\lambda^2}\Big)^{\frac{1}{2}}
+\frac{\int_\mathbb{H}|\nabla\rho(z)|\mu(dz)}{\gamma_1-\|\nabla F_n\|_{\infty}}\Big\}|x-y|\cr
&~~~~~~~~~~~~~~~~~~+e^{-\lambda t}\Bigg\}\cr
\leq&2\|f\|_{\infty}\Bigg\{\Big\{\sqrt{6C_1}\Big((\|F\|_{Lip}^2+1)\frac{\int_\mathbb{H}|z|^2\rho(z)\mu(dz)}{\lambda^2}\Big)^{\frac{1}{2}}
+\frac{\int_\mathbb{H}|\nabla\rho(z)|\mu(dz)}{\gamma_1-\|F\|_{Lip}}\Big\}|x-y|\cr
&~~~~~~~~~~~~~~~~~~+e^{-\lambda t}\Bigg\},
\end{align}
where $\{P^n_t\}_{t\geq0}$ denotes the transition semigroup of $\{X^n_t\}_{t\geq0}$.
Letting $n\rightarrow\infty$ in ($\ref{45}$),  we get
\begin{align}
\label{46}
|P_tf(x)-P_tf(y)|\leq2\|f\|_{\infty}\Big\{C_2|x-y|+e^{-\lambda t}\Big\},
\end{align}
for some constant $C_2>0$.
Since for $x,y\in\mathbb{H}$,
\begin{align*}
\mathbb{E}|X_t^x-X_t^y|\leq&|S(t)(x-y)|+\mathbb{E}\int_0^t|S(t-s)\big(F(X_s^x)-F(X_s^y)\big)|ds\\
\leq&e^{-\gamma_1t}|x-y|+\int_0^te^{-\gamma_1(t-s)}\|F\|_{Lip}\mathbb{E}|X_s^x-X_s^y|ds,
\end{align*}
then one obtain
\begin{align*}
\mathbb{E}|X_t^x-X_t^y|\leq \exp\{(-\gamma_1+\|F\|_{Lip})t\}|x-y|.
\end{align*}
Combining this with ($\ref{46}$) and using the Markov property, we have for $t>s\geq1$
\begin{align}
\label{32}
&|P_tf(x)-P_tf(y)|\cr
\leq&E|P_sf(X_{t-s}^x)-P_sf(X_{t-s}^y)|\cr
\leq&2\|f\|_{\infty}\Big\{C_2\mathbb{E}|X_{t-s}^x-X_{t-s}^y|+e^{-\lambda s}\Big\}\cr
\leq&2\|f\|_{\infty}\Big\{C_2\exp\{-(\gamma_1-\|F\|_{Lip})(t-s)\}|x-y|+e^{-\lambda s}\Big\}.
\end{align}
Let $t>\frac{(\gamma_1-\|F\|_{Lip}+\lambda)}{\gamma_1-\|F\|_{Lip}}$ and
take $s=\frac{(\gamma_1-\|F\|_{Lip})t}{\gamma_1-\|F\|_{Lip}+\lambda}$ in ($\ref{32}$),
then there exists a constant $C>0$ such that
\begin{align*}
|P_tf(x)-P_tf(y)|\leq C\|f\|_{\infty}(1+|x-y|)\exp\Big\{-\frac{\lambda(\gamma_1-\|F\|_{Lip})t}{\lambda+\gamma_1-\|F\|_{Lip}}\Big\},
\end{align*}
which implies
\begin{align*}
\|P_t(x,\cdot)-P_t(y,\cdot)\|_{Var}\leq C(1+|x-y|)\exp\Big\{-\frac{\lambda(\gamma_1-\|F\|_{Lip})t}{\lambda+\gamma_1-\|F\|_{Lip}}\Big\}.
\end{align*}
The proof is  completed by noting that the inequality trivially holds with a suitable constant $C>0$ for
$t\leq\frac{(\gamma_1-\|F\|_{Lip}+\lambda)}{(\gamma_1-\|F\|_{Lip})}$.\quad \quad $\square$\vspace{3mm}\\

\noindent$\bf{Acknowledgement}$

The authors are very grateful to Professors Zhao Dong, Yong Liu and Fengyu Wang for their valuable discussions and suggestions.

\end{document}